\documentclass[11pt,dvipdfmx]{amsart}

\usepackage{amscd}
\usepackage{amsfonts}
\usepackage{amssymb}
\usepackage{amsmath}
\usepackage{color}
\usepackage{latexsym}

\numberwithin{equation}{section}

\usepackage{hyperref}

\newtheorem{pro}{{Proposition}}[section]
\newtheorem{lemma}[pro]{{Lemma}}
\newtheorem{theorem}[pro]{{Theorem}}

\newtheorem{remark}[pro]{{Remark}}

\newtheorem{corollary}[pro]{{Corollary}}

\newtheorem{conj:intro}[pro]{Conjecture}

\newtheorem{thm:intro}{Theorem}
\newtheorem{cor:intro}{Corollary}
\newtheorem{def:intro}{Definition}

\newtheorem*{thm*}{Theorem}
\newtheorem*{thmA}{Theorem A}
\newtheorem*{thmB}{Theorem B}
\newtheorem*{thmC}{Theorem C}

\newcommand{\ra}{{\rightarrow}}
\newcommand{\lra}{{\longrightarrow}}

\newcommand{\cf}{{cf$.$\,}}

\newcommand{\cN}{{\mathcal{N}}}
\newcommand{\cM}{{\mathcal{M}}}

\newcommand{\cA}{{\mathcal{A}}}

\newcommand{\cT}{{\mathcal{T}}}

\newcommand{\SU}{\mathop{\rm SU}\nolimits}

\def\mbi#1{\boldsymbol{#1}} 

\newcommand{\al}{\alpha}
\newcommand{\be}{\beta}
\newcommand{\ga}{\gamma}

\newcommand{\om}{\omega}
\newcommand{\Om}{\Omega}
\newcommand{\we}{\wedge}
\newcommand{\lam}{\lambda}

\newcommand{\CC}{{\mathbb C}}

\newcommand{\HH}{{\mathbb H}}
\newcommand{\KK}{{\mathbb K}}

\newcommand{\RR}{{\mathbb R}}

\newcommand{\ZZ}{{\mathbb Z}}

\newcommand{\PSp}{\mathop{\rm PSp}\nolimits}

\newcommand{\PO}{\mathop{\rm PO}\nolimits}
\newcommand{\PU}{\mathop{\rm PU}\nolimits}

\newcommand{\SO}{\mathop{\rm SO}\nolimits}
\newcommand{\Sp}{\mathop{\rm Sp}\nolimits}

\newcommand{\Conf}{\mathop{\rm Conf}\nolimits}

\newcommand{\Aut}{\mathop{\rm Aut}\nolimits}

\newcommand{\Psh}{\operatorname*{Psh}\,}

\newcommand{\Diff}{\operatorname*{Diff}\,}

\newcommand{\Iso}{\mathop{\rm Iso}\nolimits}

\newcommand{\Out}{\mathop{\rm Out}\nolimits}
\newcommand{\Inn}{\mathop{\rm Inn}\nolimits}

\newcommand{\fG}{{\mathsf{G}}}

\newcommand{\sfA}{{\mathsf{A}}}

\newcommand{\sfG}{{\mathsf{G}}}

\newcommand{\sfD}{{\mathsf{D}}}

\begin{document}

\baselineskip 13pt 
\pagestyle{myheadings} 
\thispagestyle{empty}
\setcounter{page}{1}
\title[]{On the automorphism group of parabolic structures and closed aspherical manifolds}

\author[]{Oliver Baues}
\address{Department of Mathematics\\ 
University of Fribourg\\
Chemin du Mus\' ee 23\\
CH-1700 Fribourg, Switzerland}
\email{oliver.baues@unifr.ch}

\author[]{Yoshinobu Kamishima}
\address{Department of Mathematics, Josai University\\
Keyaki-dai 1-1, Sakado, Saitama 350-0295, Japan}
\email{kami@tmu.ac.jp}

\begin{abstract} 
In this expository paper we discuss several properties on closed aspherical parabolic ${\sfG}$-manifolds $X/\Gamma$. These are manifolds $X/\Gamma$,  where $X$ is a smooth contractible manifold with a parabolic ${\sfG}$-structure for which
$\Gamma\leq \Aut_{\sfG}(X)$ is a discrete subgroup acting properly discontinuously on $X$ with compact quotient. By a parabolic 
$\sfG$-structure on $X$ we have in mind a Cartan structure which is modeled on one of the classical parabolic geometries arising from simple Lie groups  $\sfG$ of rank one. Our results concern in particular the properties of the automorphism groups $\Aut_{\sfG}(X/\Gamma)$. Our main results show that the existence of certain 
parabolic ${\sfG}$-structures can pose strong restrictions on the topology of compact aspherical manifolds $X/\Gamma$ and their parabolic automorphism groups. In this realm we prove that any compact aspherical standard $CR$-manifold with virtually solvable fundamental group is diffeomorphic to a quotient of a Heisenberg manifold of complex type with its standard $CR$-structure. Furthermore we discuss the analogue properties of standard quaternionic contact manifolds in relation to the quaternionic Heisenberg group.
\end{abstract}
\date{September 22, 2023}
\thanks{This work was partially supported by JSPS grant No 22K03319}

\maketitle
\thispagestyle{empty}

\section{Introduction}
Let ${\HH}^{n+1}_{\KK}$ be the $\KK$-hyperbolic space over
$\KK = \RR, \CC$ or $\HH$. (Compare \cite{CG} and the definitions therein.)
The group of $\KK$-hyperbolic isometries $\PU(n+1,1;\KK)$
of ${\HH}^{n+1}_{\KK}$
extends to an analytic action on the boundary sphere $\partial {\HH}^{n+1}_{\KK}=S^{|\KK|(n+1)-1}$.
This corresponds to  the  {\rm rank one} standard parabolic geometry on $S^{|\KK|(n+1)-1}=\PU(n+1,1;\KK)/P_\KK$
where $P_\KK$ is a maximal parabolic subgroup.
In terms of classical geometry,  according to whether $\KK=\RR,\,\CC$ or $\HH$, the model spaces 
 $S^{|\KK|(n+1)-1}=S^n,S^{2n+1}$ or $S^{4n+3}$
admit a  \emph{conformally flat structure, spherical $CR$-structure} or
\emph{spherical quaternionic contact structure}, respectively. 
(For precise definition of these notions, see \cite{OK1} or \cite{BG}, and Section \ref{sec:parabolic_geom} of this article, for example.) \smallskip

The geometry of these spaces is intricately reflected in 
the standard graduation of the Lie algebra of $\PU(n+1,1;\KK)$, which is 
$\PO(n+1,1)$ for $\KK=\RR$ or $\PU(n+1,1)$, $\PSp(n+1,1)$ for $\KK=\CC,\HH$, where: 
\begin{equation*}\begin{split}
\mathfrak{so}(n+1,1)&= \mathfrak{g}^{-1} + \mathfrak{g}^{0} +
 \mathfrak{g}^{1}=  \RR^n + (\mathfrak{so}(n)+ \RR) + (\RR^n)^*, \\
  \mathfrak{su}(n+1,1)&= \mathfrak{g}^{-2}+ \mathfrak{g}^{-1} + \mathfrak{g}^{0} +
 \mathfrak{g}^{1} + \mathfrak{g}^{2}\\ &
 ={\rm Im}\CC +  \CC^n + (\mathfrak{u}(n)+ \RR) + (\CC^n)^* + ({\rm Im}\CC)^*, \\
\mathfrak{sp}(n+1,1)&= \mathfrak{g}^{-2}+ \mathfrak{g}^{-1} + \mathfrak{g}^{0} +
 \mathfrak{g}^{1} + \mathfrak{g}^{2}\\
&={\rm Im}\HH +  \HH^n + (\mathfrak{sp}(n)+ \mathfrak{sp}(1) + \RR) + (\HH^n)^* + ({\rm Im}\HH)^*
 \end{split}
\end{equation*} in which $\displaystyle P_\KK$ is generated by the parabolic subalgebra
$\displaystyle\mathfrak{p} =\mathfrak{g}^{0} +
 \mathfrak{g}^{1}$ for $\KK=\RR$, or $\displaystyle\mathfrak{g}^{0} +
 \mathfrak{g}^{1} + \mathfrak{g}^{2}$ for $\CC, \HH$, respectively (see \cite{Al}). \smallskip

Put $\sfG=P_\KK$, for brevity.
By a \emph{parabolic ${\sfG}$-structure} on a manifold $M$, we thus mean 
either one of a positive definite conformal structure,  
a strictly pseudo-convex $CR$-structure,  or a
positive definite quaternionic contact structure ($qc$-structure for short) on $M$.
If $\Aut_{\sfG}(M)$ is the group of structure preserving transformations on such a parabolic ${\sfG}$-manifold $M$, 
$\Aut_{\sfG}(M)$ is called:  (1) The group of conformal transformations $\Conf(M,[g])$. 
(2) The group of $CR$-transformations $\Aut_{CR}(M,\{\sfD,J\})$, or
(3) The group of $qc$-transformations $(\Aut_{qc}(M,\sfD,\,\{J_\al\}_{\al=1}^3)$, respectively.\medskip

\paragraph{\em Rigidity of parabolic  $\sfG$-manifolds with non-proper automorphism group} 
An important observation on parabolic $\sfG$-manifolds is that the
automorphism group  $\Aut_{\sfG}(M)$ does not necessarily act properly on $M$. In particular this is the case for the model spheres  $S^{|\KK|(n+1)-1}$. Given a parabolic ${\sfG}$-manifold $M$, in case $\Aut_{\sfG}(M)$ is a non-proper group 
the parabolic $\fG$-manifold $M$ is completely determined by works of D. V.\  Alekseevsky \cite{Al}, 
J.\ Ferrand \cite{Ferrand}, R.\ Schoen \cite{SC}, J.\, Lee \cite{JL},  
C.\, Frances \cite{CF}, S.\, Ivanov and D.\, Vassilev \cite{IV}, Webster \cite{WE}  and others, as follows: 

\begin{thmA}\label{thm:noncomp}
If the automorphism group $\Aut_{\sfG}(M)$ of a parabolic $\sfG$-manifold $M$ does \emph{not act properly}, then $M$ with its parabolic structure admits a structure preserving diffeomorphism  to one of the standard model spaces as specified in $(1), (2), (3) :$  
\vskip0.1cm
$(1)$ $M$ is conformal to either the standard sphere $S^{n}$ or the euclidean space $\RR^n$.
Here it  occurs
$$  \left( \Iso(M),{\rm Conf}(M) \right)= \begin{cases}
\left( {\rm O}(n+1), {\rm PO}(n+1,1) \right) & (M=S^{n}) \, ,\\
\left(\RR^n\rtimes {\rm O}(n), \RR^n\rtimes ({\rm O}(n)\times \RR^{+}) \right) & (M=\RR^n) \, .
\end{cases}
$$ 
\vskip0.2cm
$(2)$ $M$ has a spherical $CR$-structure isomorphic to
either the standard sphere $S^{2n+1}$ or the Heisenberg Lie group $\cN$ (with its canonical $CR$-structure).
It occurs
\begin{equation*}
\left({\Psh}_{CR}(M),\Aut_{CR}(M)\right)=\begin{cases}
\left({\rm U}(n+1), {\rm PU}(n+1,1)\right) \ \ \ \  (M=S^{2n+1}),\\
\left( \cN \rtimes {\rm U}(n), \cN\rtimes ({\rm U}(n)\times \RR^{+}) \right) \ (M=\cN).\, 
\end{cases}\end{equation*}
\vskip0.2cm
$(3)$
 $M$ has a spherical $qc$-structure isomorphic to
either the standard sphere $S^{4n+3}$ or the quaternionic Heisenberg nilpotent Lie group $\cM$.\\
\ $({\Psh}_{qc}(M),\Aut_{qc}(M))$ occurs 
\begin{equation*}
\begin{cases}
\left(\Sp(n+1)\cdot \Sp(1),\ {\rm PSp}(n+1,1),\ S^{4n+3}\right) \ \ \ \ \ \ \ \ \ \ \  (M=S^{4n+3}),\, \\
\left(\cM\rtimes {\rm Sp}(n)\cdot \Sp(1),\ \cM\rtimes ({\rm Sp}(n)\cdot \Sp(1)\times \RR^{+}),\ \cM\,\right)
\ (M=\cM).\,\\
\end{cases}\end{equation*}
\end{thmA}
\smallskip \vskip0.1cm

The theorem also gives the pairs $({\Psh}_{\sfG}(M),\Aut_{\sfG}(M))$ for the respective cases (1), (2), (3). Here ${\Psh}_{\sfG}(M)$ is the maximal subgroup of $\Aut_{\sfG}(M)$ that is acting properly on $M$. In particular, in case (1), ${\Psh}_{\sfG}(M)$ is the subgroup of $\Aut_{\sfG}(M))$ which preserves the canonical Riemannian metric that defines the conformal structure on $M$. In cases (2) and (3), the group ${\Psh}_{\sfG}(M)$ coincides with the subgroup of $\Aut_{\sfG}(M)$ that preserves the canonical structure defining contact forms. \smallskip

\paragraph{\em $\sfG$-Hermitian subgroups of  $\Aut_{\sfG}(M)$}

As the rigidity theorem shows, for any parabolic $\sfG$-manifold $M$,  there exists a  maximal subgroup ${\Psh}_{\sfG}(M)$ of
 $\Aut_{\sfG}(M)$ that is acting properly on $M$. 
If the parabolic structure on $M$ is of type (2) or (3), it is  defined by the conformal class of a contact form $\om$. In this case we define  ${\Psh}_{\sfG}(M, \omega)$ to be the subgroup of $\Aut_{\sfG}(M)$ that preserves $\omega$. Note that 
${\Psh}_{\sfG}(M, \omega)$ always acts properly on $M$ (see \cite{OK1}), that is, ${\Psh}_{\sfG}(M, \omega)$ is contained in ${\Psh}_{\sfG}(M)$. 
The groups ${\Psh}_{\sfG}(M, \omega)$ are also called  $\sfG$-Hermitian subgroups of  $\Aut_{\sfG}(M)$, see
 Section \ref{sec:parabolic_geom} below. As to the precise relation of ${\Psh}_{\sfG}(M)$ with the $\sfG$-Hermitian subgroups
we additionally have the following:  \smallskip
\begin{thm:intro}[see Proposition \ref{prop:vanish}, \cite{OK1}]\label{thm:proper}
Let $M$ be a parabolic $\sfG$-manifold and let $H \leq \Aut_{\sfG}(M)$  be a closed subgroup. Then there exists a canonical cohomology class  
$[\lam_{\sfG}]$ in the differentiable cohomology group $H^1_d(H,C^\infty(M,\RR^+))$ 
with the following properties:
\begin{enumerate}
\item[i)]	 If\/ $[\lam_{\sfG}]=0$ and the
 $\fG$-structure is conformal, then there exists a Riemannian metric $g$ representing the conformal structure such that $H$ is contained in $\Iso(M,g)$. 
\item[ii)] If\/ $[\lam_{\sfG}]=0$ and the parabolic structures is of type $(2)$ or $(3)$, there exists a compatible contact form $\omega$ 
(that is, $\om$ is representing the parabolic structure) such that $H$ is contained in 
$\displaystyle {\Psh}_{\sfG}(M,\omega)$.
\item[iii)] The group $H$ acts properly on $M$ if and only if  $[\lam_{\sfG}]=0$.  
\end{enumerate}
\end{thm:intro}   \smallskip     

If $M$ is compact, Theorem \ref{thm:proper} combined  with Theorem A 
implies that $\Aut_{\sfG}(M)$ is a compact Lie group, except if $M$ is one of the standard parabolic $\sfG$-spheres as described in Theorem A. \smallskip

In the following we will be in particular interested in compact  parabolic $\sfG$-manifolds. 
\smallskip 

\paragraph{\em Automorphisms of aspherical parabolic $\sfG$-manifolds} 
A compact manifold $M$ is called aspherical if its universal covering manifold $X$ is contractible. In that case $\Aut_{\sfG}(M)$ is compact and its identity component  $\Aut_{\sfG}(M)^0$ is a compact torus. The latter fact is a consequence of the  following fundamental result on compact  Lie group actions on closed aspherical manifolds: 

\begin{thmB}[Conner and  Raymond  \cite{CR1,CR}]\label{thm:2,3}
Let $X/\Gamma$ be a closed 
aspherical Riemannian manifold. Then the isometry group $\Iso(X/\Gamma)$ is a finite group or
the identity component
$\Iso(X/\Gamma)^0$ is isomorphic to a $k$-torus $T^k$.
Moreover, there is a central group extension\eqref{Pshqcr},
$1\ra \ZZ^k\ra\, \Gamma\lra\, Q\ra 1$, where $Q=\Gamma/\ZZ^k$.
\end{thmB}

Theorem B shows that if  the fundamental group $\Gamma$
of $M$ has no normal solvable subgroup then $\Aut_{\sfG}(M)$ is a 
finite group. Well known examples of such aspherical  manifolds $M$ are
compact locally symmetric spaces of non-compact type (without local flat factors).  In this context, we remark the following general fact:

\begin{thm:intro} \label{thm:1,2,3}
Let $X/\Gamma$ be a closed aspherical manifold such that $\Gamma$
has no normal solvable subgroup. If $X/\Gamma$ admits a parabolic ${\sfG}$-structure, then its automorphism 
$\Aut_{\sfG}(X/\Gamma)$ is a finite group which is isomorphic to a subgroup of $\Out(\Gamma)$.
In particular, $\Gamma$ is of finite index in its normalizer $N_{\Aut_{\sfG}(X)}(\Gamma)$. 
\end{thm:intro}
\vskip0.1cm
In the theorem $\Out(\Gamma) = \Aut(\Gamma)/ \Inn(\Gamma)$ denotes the outer automorphism group of $\Gamma$.\medskip

\paragraph{\em $CR$- and $qc$-structures on closed aspherical manifolds}\  \smallskip

From our viewpoint of \emph{parabolic ${\sfG}$}-structures, we are interested how the existence of a parabolic ${\sfG}$-structure determines the topology of $X/\Gamma$ and in particular its smooth structure. We show that the existence of certain  $CR$- or $qc$-parabolic structures  on $X/\Gamma$, that admit a non-trivial connected group of automorphisms $\Aut_{\sfG}(X/\Gamma)^0$,  poses a strong restriction on $\Gamma$. 
In fact, we show that under the assumption that $\Gamma$ is virtually solvable any standard $CR$-manifold   $X/\Gamma$ is \emph{diffeomorphic} to a Heisenberg type manifold which is derived from the standard model space in (2) of Theorem A. 
In the $qc$-case a much stronger rigidity result holds, namely we show that \emph{any} closed aspherical standard $qc$-manifold $X/\Gamma$ is \emph{$qc$-equivalent} to a $qc$-manifold of quaternionic Heisenberg type (as in (3) of Theorem A).  
\smallskip

\paragraph{\em Heisenberg manifolds} 

Recall from Theorem A 
that the Heisenberg Lie group $\cN$ admits a maximal proper subgroup of affine transformations $\cN\rtimes {\rm U}(n)$. 
The group $\cN\rtimes {\rm U}(n) = {\Psh}_{CR}(\cN)$ then preserves the canonical left-invariant $CR$-structure, and the canonical pseudo-Hermitian structure on the Heisenberg group $\cN$ (see for example, \cite{Ka1,OK}). Given a torsion-free discrete uniform subgroup $\Gamma$
contained in $\cN\rtimes {\rm U}(n)$, the compact quotient manifold $$M= \cN/\Gamma$$ is then called a \emph{Heisenberg infra-nilmanifold}. 
 \smallskip

\paragraph{\em Standard $CR$-structures} Suppose that $M$ admits a $CR$-structure with contact form $\om$, where $M$ is a $2n+1$-dimensional manifold. Then the $CR$-structure is called \emph{standard} if the  Reeb vector field associated to $\om$ 
generates a one-parameter subgroup of ${\Psh}_{CR}(M, \omega)$. \smallskip

In that sense, every Heisenberg infra-nilmanifold carries a canonical standard $CR$-structure,  
where $\Aut_{CR}(M)^0 = \Psh_{CR}(M, \omega)^0=S^1$.
(The structure is induced from the standard $CR$-structure on the Heisenberg group $\cN$.) See Section \ref{sec:CRsolv}
for details. \smallskip

Further examples of (aspherical) standard $CR$-manifolds may be constructed as $S^1$-bundles over any compact Hermitian locally symmetric spaces $B$, using the K\"ahler class of $B$ to determine the circle bundle. (In fact, this  construction works over every compact K\"ahler manifold $B$. See \cite{OK} and the references therein). \smallskip

\paragraph{\em Virtually solvable fundamental group $\Gamma$}
Let $M= X/\Gamma$ be a  closed aspherical manifold such that $\Gamma$ is a virtually solvable group (which  means that $\Gamma$ contains a solvable subgroup of finite index).  
In fact, since $M$ is aspherical, $\Gamma$ is a torsion-free virtually polycyclic group. It is known that every such group occurs as the fundamental group of a 
compact aspherical manifold $X/\Gamma$. Note further that the fundamental group  $\Gamma$ determines $M$ up to homeomorphism, but not necessarily up to diffeomorphism unless some further geometric structure on $M$ is specified that enforces smooth rigidity  (\cf \cite{Baues} and the references therein.) \smallskip 

Every \emph{standard} $CR$-manifold
with solvable fundamental group turns out to be diffeomorphic to a Heisenberg infra-nilmanifold:

\begin{thm:intro}\label{Tcr}
Let $M = X/\pi$ be a $2n+1$-dimensional closed
aspherical positive definite strictly pseudoconvex \emph{standard} $CR$-manifold. If $\pi$ is virtually solvable then then there exists a discrete faithful representation
$\rho : \pi\ra\, \cN\rtimes {\rm U}(n)$ and $M$ is diffeomorphic to the Heisenberg infra-nilmanifold 
$\cN/\rho(\pi)$. 
In particular, $\pi$ is virtually nilpotent and $\Aut_{CR}(M)^0=S^1$.
\end{thm:intro}

\smallskip

By choosing a representative contact form $\om$ for the $CR$-structure on $M$, we obtain a pseudo-Hermitian manifold $(M,(\om,J))$.  Note that a standard pseudo-Hermitian manifold $(M,(\om,J))$ is equivalent to a Sasaki manifold $(M, g, (\om,J),\xi)$ by assigning a positive definite Riemannian metric $g=\om\cdot \om+d\om\circ J$, 
called Sasaki metric. (Compare \cite{OK,Bl}.)\smallskip

By the existence of the $S^1$-action generated by the Reeb field,
we see that $M$ admits a fibering over a K\"ahler orbifold. We will prove:

\vskip0.2cm
\noindent{{\bf Theorem 3$'$}. }
{\em If the fundamental group of a closed aspherical Sasaki manifold $M$
is virtually solvable, then $M$ is diffeomorphic to a Heisenberg infra-nilmanifold.
}\vskip0.2cm

A Sasaki manifold is called \emph{regular} if the $S^1$-action generated by the Reeb field is free. For a regular Sasaki manifold 
Theorem $3'$ is stated and proved in  \cite[Corollary 2, Proposition 6.10]{OK}. Assuming that the fundamental group of $M$ is nilpotent,  a proof of Theorem 3$'$  involving methods of rational homotopy theory is provided in \cite{NY}. \smallskip

The following result is
obtained by O. Baues and V. Cort\'es \cite{BC} which is used for our proof.

\begin{thmC}\label{Tconf}
Let $X/\Gamma$ be a closed aspherical K\"ahler manifold with $\Gamma$ virtually solvable. Then a finite cover of $X/\Gamma$ is biholomorphic to a complex torus.
\end{thmC}

We next seek whether similar results hold for  $4n+3$-dimensional $qc$-manifolds. 
\smallskip

\paragraph{\em Quaternionic Heisenberg manifolds}
A quaternionic Heisenberg Lie group 
is a $4n+3$-dimensional nilpotent Lie group  $\mathcal M$ with  
center $\RR^3$ whose quotient is isomorphic to
the quaternionic vector space $\HH^n$, whose structure determines the Lie product on $\cM$ (see \cite{Al,Ka}). The group 
$\cM$ admits a maximal proper subgroup of affine transformations
 ${\Psh}_{qc}(\cM)=\cM\rtimes \Sp(n)\cdot \Sp(1)$. Then ${\Psh}_{qc}(\cM)$ preserves the canonical $qc$-structure 
 and the canonical $qc$-Hermitian structure on $\cM$.
 (Compare Theorem A.) A quotient $\cM/\Gamma$
by a torsion-free discrete uniform subgroup $\Gamma\leq {\Psh}_{qc}(\cM)$ is called a quaternionic Heisenberg infra-nilmanifold. 
\smallskip

For any $qc$-manifold $M$, let $\cT =\{\xi_1,\xi_2,\xi_3\}$ denote the three-dimensional integrable distribution complementary to the   codimension three subbundle $\sfD$ on $M$ determined by the $qc$-structure, called \emph{$qck$-distribution} (see Section \ref{sec:parabolic_geom}). If $\cT$ generates a subgroup
of ${\Psh}_{qc}(M)$, then $M$ is called a \emph{standard} $qc$-manifold. \smallskip

\paragraph{\em Remark} When $M$ is a closed aspherical standard $qc$-manifold, 
$\cT$ generates a three-torus $T^3\leq
{\Psh}_{qc}(M)$.  
In this case, the $qc$-structure on $M$  thus does not give rise to a $3$-Sasaki structure since the $qck$-distribution $\cT$ does not generate an $\Sp(1)$-action on $M$ but a $T^3$-action \cite[Theorem 5.4]{Ka}. Then the  quotient space $M/\cT$ inherits a hyper-K\"ahler structure from $M$.  \smallskip

We have  the following strong rigidity property for standard $qc$-structures on closed aspherical manifolds. 

\begin{thm:intro}\label{Tqc}
Let $X/\pi$ be a positive definite closed aspherical \emph{standard} $qc$-manifold.
Then $X/\pi$ is $qc$-isometric to
a quaternionic Heisenberg infra-nilmanifold $\cM/\pi$
where $\pi\leq \cM\rtimes \PSp(n)$, 
is a discrete uniform subgroup.
In particular, $\pi$ is virtually nilpotent and $\Aut_{qc}(X/\pi)^0=T^3$ is a three-torus. 
\end{thm:intro}

In the context of hyper-K\"ahler manifolds the following striking fact plays an important role in the proof of Theorem \ref{Tqc}:  \smallskip

\paragraph{\em Rigidity of aspherical hyper-K\"ahler manifolds}Recall that  the quaternionic torus $T^n_\HH$ admits a natural flat homogeneous hyper-K\"ahler structure which is induced by a linear quaternionic hermitian form on $\HH^n$ (\cf \cite{BG}). We then have:
\vskip0.3cm
\noindent{\bf Lemma  D.}\ 
{\em Let $M$ be a closed aspherical hyper-K\"ahler manifold.
Then a finite cover of 
$M$ is \emph{hypercomplexally isometric} to  
a quaternionic torus $T^n_\HH$ with its natural flat hyper-K\"ahler structure. }
\vskip0.1cm
Here a hypercomplex isomorphism is simultaneously a holomorphic diffeomorphism with respect to each complex structure. This strong rigidity for aspherical hyper-K\"ahler manifolds is a consequence of the Calabi-Yau theorem and the Cheeger-Gromoll splitting theorem. For related result on complex hyperhermitian surfaces, see \cite{Bo}. \smallskip

\paragraph{\em The paper is organized as follows}
In Section 2, we prove Theorem \ref{thm:1,2,3} of the Introduction saying that
$\Aut_{\sfG}(X/\Gamma)$ is finite whenever $\Gamma$ has no normal solvable groups.  
In Section 3, we introduce a cohomology-invariant for the group $\Aut_{\sfG}(M)$ (see Proposition \ref{prop:vanish}) to
show the coincidence of the Lie groups $\Aut_{\sfG}(M)$ and ${\Psh}_{\sfG}(M)$
when $\Aut_{\sfG}(M)$ acts properly on $M$ in Theorem 
\ref{th:equal}.
In particular, when $M$ is compact,
$\Aut_{\sfG}(M)$ is a compact Lie group. 
We prove Theorem \ref{Tcr} in Section \ref{sec:CRsolv}.
Section \ref{sec:qcsolv} concerns the properties of   standard $qc$-manifolds. 


\section{${\sfG}$-manifolds with compact automorphism groups}
\label{sec:compact_auto} 
We study the structure of $\Aut_{\sfG}(X/\Gamma)$ for closed aspherical ${\sfG}$-manifolds.\smallskip

\paragraph{\em Lifting Lemma} To prepare our proof we start our discussion with some well known general setup.  Let $\tilde M$ be the universal covering space of $M=\tilde M/\pi$ and
denote $N_{{\rm Diff}(\tilde M)}(\pi)$ the normalizer of $\pi$ in
${\rm Diff}(\tilde M)$. The conjugation $$\displaystyle
\mu: N_{{\rm Diff}(\tilde M)}(\pi)\ra{\rm Aut}(\pi)$$ 
defined by
$\mu(\tilde f)(\ga)=\tilde f\circ\ga\circ \tilde f^{-1}$
$({}^\forall\,\ga\in \pi)$ induces a homomorphism 
$$ \varphi : 
{\rm Diff}(M) \to {\rm Out}(\pi) \; .$$  

\begin{lemma}[see \cite{LR}] 
\label{lem:lift1_diagram}
There is an exact commutative
diagram:
\vskip0.1cm
\begin{equation}\label{lift1}
\begin{CD}
 @. 1@. 1 @. 1\\
@.  @VVV @VVV @VVV \\
 1@>>> Z(\pi)@>>>  \pi @>\mu>> {\rm Inn}(\pi)\\
 @. @VVV  @VVV @VVV\\
1@>>>Z_{{\rm Diff}(\tilde M)}(\pi)@>>> N_{{\rm Diff}(\tilde M)}(\pi)
 @>\mu>> {\rm Aut}(\pi)\\
@. @V\nu VV @V\nu VV @VVV\\
1@>>>{\rm ker}\,\varphi@>>>{\rm Diff}(M)
@>\varphi>> {\rm Out}(\pi)\\
@. @VVV @VVV @VVV\\
 @. 1@. 1 @. 1\\
\end{CD}
\end{equation} 
\end{lemma} 
Here, $Z_{{\rm Diff}(\tilde M)}(\pi)$
denotes the centralizer of
$\pi$ in ${\rm Diff}(\tilde M)$. \medskip

\paragraph{\em Application to automorphisms of aspherical manifolds}

\begin{proof}[{\bf Proof of Theorem $\ref{thm:1,2,3}$}]

As we show now,  replacing ${\rm Diff}(\tilde M)$ in Lemma \ref{lem:lift1_diagram} by $\Aut_{\sfG}(X)$,
$N_{\Aut_{\sfG}(X)}(\Gamma)$ turns out to be
discrete and the homomorphism 
$\mu: N_{\Aut_{\sfG}(X)}(\Gamma)\ra {\rm Aut}(\Gamma)$ injective.  

As in diagram \eqref{lift1}, there is an exact sequence:
\begin{equation}\label{ex:1}
1\ra \ker \varphi\ra \Aut_{\sfG}(X/\Gamma)\stackrel{\varphi}\lra \Out(\Gamma)
\end{equation} such that $\Aut_{\sfG}(X/\Gamma)^0\leq \ker \varphi$.    
Let $Z_{\Aut_{\sfG}(X)}(\Gamma)$ be the centralizer of $\Gamma$ in $\Aut_{\sfG}(X)$. 
Then as can be inferred from diagram \eqref{lift1}, $\ker \varphi$ is associated with the covering group extension:
\[1\ra\, Z(\Gamma)\ra\, Z_{\Aut_{\sfG}(X)}(\Gamma)\lra \, \ker \varphi \ra 1.\]
Since $Z(\Gamma)$ is the center of $\Gamma$,
note $Z(\Gamma)=\{1\}$ by the hypothesis so that 
\begin{equation}\label{ex:2}
 Z_{\Aut_{\sfG}(X)}(\Gamma)\cong \ker \varphi.
 \end{equation}
 
Now 
$\Aut_{\sfG}(X/\Gamma)$ is compact by Corollary \ref{eq:Pcomp}. 
Since a compact connected Lie group acting on a closed aspherical manifold is a torus (\cf Theorem B),
$\Aut_{\sfG}(X/\Gamma)^0=T^k$,  $k\geq 0$.  
Then the orbit map ${\rm ev}(t)=tz$ of $T^k$ into $X/\Gamma$ at any point $z\in X/\Gamma$ 
induces an injective homomorphism 
 $\displaystyle {\rm ev}_*: \ZZ^k\lra\, \Gamma$ 
such that ${\rm ev}_*( \ZZ^k)\leq Z(\Gamma)$ (\cf \cite[Lemma 4.2]{CR}). 
 As $Z(\Gamma)$ is trivial, by assumption, we have $k=0$.  That is, $\Aut_{\sfG}(X/\Gamma)^0=\{1\}$. In particular $\Aut_{\sfG}(X/\Gamma)$ is a finite group. 
 
Therefore, the subgroup $\ker \varphi$ is finite by \eqref{ex:1}, and  so is $Z_{\Aut_G(X)}(\Gamma)$ by \eqref{ex:2}.
As $Z_{\Aut_{\sfG}(X)}(\Gamma)$ is centralized by $\Gamma$,
it follows $\displaystyle Z_{\Aut_{\sfG}(X)}(\Gamma)=\{1\}$ by \cite[Theorem 2]{BK}. 
Hence $\ker \varphi=\{1\}$, that is $\displaystyle \varphi : \Aut_{\sfG}(X/\Gamma)\ra \Out(\Gamma)$ is injective. 

Since the quotient of $N_{\Aut_{\sfG}(X)}(\Gamma)$ by $\Gamma$ 
is isomorphic to $\Aut_{\sfG}(X/\Gamma)$ from \eqref{lift1}.
As $\Aut_{\sfG}(X/\Gamma)$  is finite, $\Gamma$ is of finite index in $N_{\Aut_{\sfG}(X)}(\Gamma)$. \end{proof}

Along the same direction as the above proof,  we have: 
\begin{corollary}\label{cor:Liegroup}
Let $X/\Gamma$ be a closed aspherical manifold.
Suppose that $G\leq \Diff(X/\Gamma)$ is a connected Lie group.
If $Z(\Gamma)$ is trivial (or finite), then $G$ is a simply connected solvable Lie group.
\end{corollary}

\section{Invariant substructures for  ${\Psh}_{\sfG}(M)$} 
\label{sec:parabolic_geom} 
Let $M$ be a parabolic $\sfG$-manifold. If $M$ is of $CR$- or $qc$-type,  we shall prove that, when $\Aut_{\sfG}(M)$ acts properly, there exists a representative form $\om$ for the  parabolic $\sfG$-structure such that $\Aut_{\sfG}(M)$ coincides
with ${\Psh}_{\sfG}(M,\omega)$. (The obvious analogue also holds for conformal structures defined by Riemannian metrics. See \cite{OK1} and references therein for both results).\smallskip

\paragraph{\em Geometry associated with parabolic $\sfG$-structures} Whereas a conformal structure is equivalent with a conformal class of Riemannian metrics, the classical geometries  underlying the case (2) and (3) parabolic geometries are considerably more involved. Let us thus briefly recall the geometric data associated with case (2) and (3) parabolic geometries: \smallskip 

In the case of $CR$-structures, 
we have a contact form $\omega$ on a connected smooth manifold $M$, which is  determined up to scaling with a positive function,  and a complex structure $J$ on the contact bundle $\ker \om$ which is  compatible with $\om$ in the sense that the Levi form $d\om  \circ J$ is 
 a positive definite Hermitian form on $D$. These data 
define a \emph{strictly pseudo-convex} $CR$-structure. Note that $\om$ is defined up to a conformal change with a positive function.
\smallskip 

Let a $qc$-structure on a $4n+3$-manifold $M$ be given.  This amounts
to a positive definite codimension three subbundle $\sfD$, which is non-integrable and such that $\sfD+ [\sfD,
\sfD]=TM$. Moreover, there is a hypercomplex structure 
$\{J_k\}_{k=1}^3$ on $\sfD$, and 
an ${\rm Im}\,\HH$-valued $1$-form $\om =\om_1 i+ \om_2 j + \om_3 k$.
It is also required that  $\sfD=\mathop{\ker}\,\om$ and the forms $d\om_k  \circ J_k$ are positive definite Hermitian forms. 
Note that $\om$ is defined up to a conformal change with a positive function and conjugation with $\Sp(1)$.
\medskip

\paragraph{\em Description of the automorphism groups of parabolic $\sfG$-structures}
The associated automorphism groups $\Aut_{\sfG}(M)$ are then described as follows (for cases (2) and (3) compare
\cite{OK1,Ka}):  
\begin{equation}\label{Autqcr}
\begin{cases}
(1)\,\Conf(M)=\bigl\{\, \al\in\Diff(M)\mid \al^*g=u_\al\,g \,\bigr\},\\
(2)\,\Aut_{CR}(M, \{\om,J\}) =\{ \, \al\in\Diff(M)\mid  \alpha^*\om=u_\al\, \om,  \\
 \ \ \ \ \ \ \ \ \ \ \ \ \ \ \ \ \ \ \ \ \ \ \ \ \ \ \ \ \ \ \ \ \  \ \  \alpha_*\circ J=J\circ \alpha_*|_{\ker\, \om} \, \bigr\}, \\
(3) \Aut_{qc}(M, (\om,\{J_k\}_{k=1}^3))=\bigl\{\, \al\in\Diff(M)\mid  \\ 
\ \ \ \ \ \ \ \ \ \ \ \al^*\om =u_\al \; a_\al\cdot\om\cdot\overline{a_\al}, 
 \ \ \al_* \circ J_k=\sum_{j=1}^{3}a_{kj}J_j\circ \al_*|_{\ker\, \om} \,\bigr\}, \\
\end{cases}
\end{equation}
where $u_\al\in C^\infty(M,\RR^+)$, 
$a_\al\in C^\infty(M,\Sp(1))$, and 
the matrix $(a_{kj})\in C^\infty(M,\SO(3))$ is given by the conjugation action of $a_\al$ on ${\rm Im}\,\HH$. \smallskip
  
We would like to emphasize that the definition of $\Aut_{\sfG}(M)$ does not depend on the particular choice of data $g$ or $\omega$ in their conformal class. In fact, the choice of $g$ or $\om$ amounts to choosing a representative geometry. The symmetries of the representative geometry define a subgroup of $\Aut_{\sfG}(M)$. These groups are the isometry group $\Iso(M,g)$, respectively the pseudo-Hermitian groups ${\Psh}_{CR}(M, \{\om,J\})$ and ${\Psh}_{qc}(M, (\om,\{J_k\}_{k=1}^3))$:

\begin{equation}\label{Pshqcr}
\begin{cases}
(1)\,\Iso(M,g)=\bigl\{\al\in\Diff(M)\mid \al^*g=g\ \bigr\},\\
(2)\, {\Psh}_{CR}(M, \om ) =\{\al\in\Diff(M)\mid \alpha^*\om=\om, \\
\ \   \ \ \ \ \ \ \ \ \ \ \ \ \ \ \ \ \ \ \ \ \ \ \ \ \ \ \ \ \ \ \alpha_*\circ J=J\circ \alpha_*|_{\ker\, \om}\bigr\}, \\
(3)\, {\Psh}_{qc}(M, \om)= 
\bigl\{\al\in\Diff(M)\mid  \al^*\om=a_\al\cdot \om\cdot \overline{a_\al},\\
\, \ \ \ \ \ \ \ \ \ \ \ \ \ \ \ \ \ \ \ \ \ \ \ \ \ \ \ \ \ \ \ \al_* \circ J_k=\sum_{j=1}^{3}a_{kj}J_j\circ \al_*|_{\ker\, \om} \bigr\}. \\
\end{cases}
\end{equation}

Note that the groups in \eqref{Pshqcr} vary considerably under a  conformal change of $g$, respectively $\om$, while the group $\Aut_\sfG(M)$ is preserved. 

\subsection{Conformal invariant cohomology class}\label{Confinv}

The space $C^\infty(M,\RR^+)$ of smooth positive functions on $M$ is endowed with an action of $\Aut_{\sfG}(M)$, where for $\al\in \Aut_{\sfG}(M)$, $f\in C^\infty(M,\RR^+)$, we have  
\[
(\al_* f)(x)=f(\al^{-1}x)\,\ (x\in M).\]
Thus $C^\infty(M,\RR^+)$ is a smooth $\Aut_{\sfG}(M)$-module. To any such module there is an associated differentiable group cohomology $H^*_d$ for the Lie group $\Aut_{\sfG}(M)$ (see \cite{OK1}, for detailed explanation). We explain now that the action of $\Aut_{\sfG}(M)$ on  $C^\infty(M,\RR^+)$ gives rise to a natural cohomology class that carries geometric information about the dynamics of this action.  \smallskip 

\paragraph{\em Construction of the associated cohomology class } 

\begin{pro}\label{prop:vanish}
For any closed subgroup $L$ of $\Aut_{\sfG}(M)$ there is a natural cohomology class $[\lam_{\sfG}]\in H^1_d(L,C^\infty(M,\RR^+))$, 
which is associated to the parabolic $\sfG$-structure on $M$. 
\end{pro}

\begin{proof}
Let $\al\in \Aut_{\sfG}(M)$ such that $\displaystyle \al^*g=u_\al g$
for a Riemannian metric,
$\al^*\om=u_\al \om$ for a contact form, or $\displaystyle 
\al^*\om=u_\al\,   a_\al \cdot \om \cdot\overline{a_\al}$ for a quaternionic contact form, representing the parabolic $\sfG$-structure on $M$. We construct $[\lambda_\sfG]$ for the $qc$-group $\Aut_{qc}(M)$ in place of $\Aut_{\sfG}(M)$. (But the proof holds also for $\Conf(M)$, $\Aut_{CR}(M)$, see \cite{OK1})

Let $\al, \be\in \Aut_{qc}(M)$. We write $\alpha \, \beta \in \Aut_{qc}(M)$ for the composition of the two $qc$-transformations. We calculate
\begin{equation*}\begin{split}
(\al \, \be)^*\om&=u_{\al\be}\; a_{\al\be} \cdot   \om \cdot\overline{a_{\al\be}},\\
\be^*\al^*\om&= \be^*(u_\al\,  a_\al \cdot \om\cdot \overline{a_\al})
=(\be^*u_\al\,  u_\be) (\be^*a_\al\cdot  a_\be) \cdot \om\cdot
(\overline{\be^*a_\al\cdot a_\be}). \\
\end{split}
\end{equation*}
(Note that $\displaystyle \be^* \,\overline{a_\al}\, (x)=
\overline{a_\al}\,(\be x)=\overline{a_\al(\be x)}=
\overline{\be^*a_\al}\,(x)$).
Taking the norm, we have
$$ \displaystyle ||(\al\be)^*\om||=u_{\al\be}\, ||\om||=\be^*u_\al\,  u_\be\, ||\om|| \, .$$

Thus the smooth maps $u_\al, u_\be, u_{\al\be}\in C^\infty(M,\RR^+)$ satisfy
\begin{equation}\label{crossed}
 u_{\al\be}=\be^*u_\al \, u_\be\ \ \mbox{on}\ M. 
\end{equation}

Define $\lam_{\sfG} = \lam_{\sfG,\omega}: \Aut_{\sfG}(M)\to C^\infty(M,\RR^+)$ to be
\begin{equation}\label{crossedhom}
\lam_{\sfG} (\al)=\al_*u_\al.
\end{equation}
In particular, $\displaystyle \lam_{\sfG}(\al)(x)=u_\al(\al^{-1}x)$.
We observe that $\lam_{\sfG}$ is a crossed homomorphism with respect to the representation of $\Aut_{\sfG}(M)$ on $C^\infty(M,\RR^+)$:
\begin{equation*}
\begin{split}
\lam_{\sfG}(\al\be) \, (x)&= \, (\al\be)_*u_{\al\be}\, (x)=u_{\al\be} \, (\be^{-1}\al^{-1}x)\\
 &=\be^*u_\al\, (\be^{-1}\al^{-1}x)\;  u_\be\, (\be^{-1}\al^{-1}x)\ \, (\text{by} \eqref{crossed})\\
&= \lam_{\sfG}(\al)\, (x)\; \al_*\lam_{\sfG}(\be)\, (x)
=(\lam_{\sfG}(\al)\;\al_*\lam_{\sfG}\, (\be))\, (x).
\end{split}\end{equation*}
Hence, 
$\displaystyle \lam_{\sfG}(\al\be)=\lam_{\sfG}(\al)\cdot\al_*\lam_{\sfG}(\be)$,
that is, $\lam_{\sfG}$ is a crossed homomorphism and thus a one-cocycle for the differentiable cohomology of $\Aut_{\sfG}(M)$ with coefficients in $C^\infty(M,\RR^+)$. Let 
$\displaystyle [\lam_{\sfG}]\in H^1_d(\Aut_{\sfG}(M),\,C^\infty(M,\RR^+))$ denote its corresponding cohomology class.

We show $[\lam_{\sfG}]$ is a conformal $\sfG$-invariant.
For $qc$-forms $\om, \om'$, suppose that $\om'$ is $qc$-equivalent to $\om$, that is, 
\begin{equation}\label{conG}
\om'=u \; b\cdot \om  \cdot \bar  b, \;  \; \;  \text{ for } u\in C^\infty(M,\RR^+), \;  b \in C^\infty(M,\Sp(1)) \; . 
\end{equation} 
For $\al\in \Aut_{\sfG}(M)$, write
$\displaystyle \al^*\om'=u'_\al\; a'_{\al} \cdot \om'\cdot \overline{a'_{\al}}$.
Thus  $$\lam_{\sfG, \om'} (\al)=\al_*u'_\al \, .$$
Then $\displaystyle \al^*\om'=u'_\al\;  a'_{\al} \cdot  (u \; b\cdot \om \cdot \bar b) \cdot  \overline{a'_{\al}}
=(u'_\al u)\;  (a'_{\al}\cdot b) \cdot \om\cdot (\overline{a'_{\al} \cdot b})$. Also
\begin{equation*}\begin{split}
\al^*\om'&=\al^*(u\;  b\cdot \om \cdot \bar b)
=(\al^*u\,  u_\al)\,  (\al^*b \cdot a_\al) \cdot \om\cdot
(\overline{\al^*b\cdot a_\al}).
\end{split}\end{equation*}
Taking the norm $||\al^*\om'||$, it follows
$\displaystyle u'_\al\, u=\al^*u\, u_\al$, that is, 
$$ \displaystyle u'_\al \,(\al^*u)^{-1} \,  u =   u_\al .$$
This shows
$\displaystyle \al_*u'_\al\cdot \delta^0(u)(\al)=\al_*u_\al$. 
Hence, 
$\displaystyle [\lam_{\sfG, \om'}]=[\lam_{\sfG,\om}]$ and so the cohomology class 
$[\lam_{\sfG, \omega}]$ is a quaternionic conformal invariant. 
\end{proof}

Regarding the cohomology groups of $L$ with coefficients in  $C^\infty(M,\RR^+)$ we have the following important general fact: 

\begin{theorem}[\mbox{\cite[Theorem 10]{OK1}}] \label{thm:vanish} 
Suppose that $L$ acts properly on $M$. Then $$H^i(L,  C^\infty(M,\RR^+) )= \{0 \} , \; i \geq 1 .$$
\end{theorem} \smallskip

Recall that ${\Psh}_{\sfG}(M)$ denotes the unique maximal subgroup of $\Aut_{\sfG}(M)$ that acts properly on $M$. We now prove: 

\begin{theorem}\label{th:equal}
Let $M$ be a parabolic $\sfG$-manifold of $CR$- or $qc$-type.  
Then there exists a representative form $\om$ for the  parabolic $\sfG$-structure such that 
$${\Psh}_{\sfG}(M)= {\Psh}_\sfG(M,\om) \; . $$ 
\end{theorem}

\begin{proof} Put $L = {\Psh}_{\sfG}(M)$ for the following. 
Since ${\Psh}_{\sfG}(M)$ acts properly, Theorem \ref{thm:vanish} implies $H^1(L , C^\infty(M,\RR^+)) = \{ 0 \}$. Let $\eta$ be a representative form for the parabolic $\sfG$-structure on $M$. 
Since $H^1(L , C^\infty(M,\RR^+)) = \{ 0 \}$, in particular, $[\lam_{\sfG,\eta}]=0$, where $\lam_{\sfG,\eta}$ is a one-cocycle for $L$. 

Since $\alpha \in \Aut_{\sfG}(M)$, we can write $\al^*\eta = u_\al\,   a_\al\cdot \eta \cdot \overline{a_\al}$. 
Thus the equation  $$\lam_{\sfG,\eta}=\delta^0 v,   \text{ for some } v\in C^\infty(M,\RR^+), $$ means that $\displaystyle \al_*u_\al =\al_*v v^{-1}$, $\alpha \in L$. Or equivalently,  $\displaystyle u_\al\cdot \al^*v=v$.

Put  $\omega =v\, \eta$.
Then it follows that
$$\displaystyle \al^*\omega =\al^*v\;   u_\al\, a_\al\cdot \eta \cdot\, \overline{a_\al}
=v\;  a_\al\cdot \eta \cdot \overline{a_\al}=a_\al\cdot \omega \cdot \overline{a_\al}.$$
This shows that $\alpha \in {\Psh}_{\sfG}(M,\om)$. That is,
we have ${\Psh}_\sfG(M)$ is contained in ${\Psh}_{\sfG}(M,\om)$.
Since $\Aut_{\sfG}(M,\om)$  is acting properly, by Lemma \ref{lem:inv_metric}, $\Aut_{\sfG}(M,\om)$ 
is  contained in ${\Psh}_\sfG(M)$. The theorem is proved. 
\end{proof}

Next we note: 

\begin{lemma} \label{lem:inv_metric} 
The subgroup ${\Psh}_{qc}(M, \om)$ of $\Aut_\sfG(M)$ preserves an associated Riemannian metric $g_\om$. In particular, ${\Psh}_{qc}(M, \om)$ acts properly on $M$. 
\end{lemma}
\begin{proof}
\begin{equation*}
	g_\om(\mbi{x},\mbi{y})=\sum_{i=1}^{3}\om_i(\mbi{x})\cdot \om_i(\mbi{y})+
d\om_1(J_1 \mbi{x}, \mbi{y}), \, \text{ where } \mbi{x}, \mbi{y}\in TM. \qedhere \end{equation*}
\end{proof}

\begin{corollary}\label{eq:Pcomp}
When $X/\Gamma$ is a closed aspherical $\sfG$-manifold,
$\Aut_{\sfG}(X/\Gamma)$ is a compact Lie group. In particular, $[\lam_{\sfG}]=0$.
\end{corollary}

\begin{proof}
By Theorem A, the automorphism group of any compact aspherical parabolic $\sfG$-manifold is acting properly. 
\end{proof}

\begin{remark}
Note that, as $\Aut_{CR}(\cN)=\cN\rtimes {\rm U}(n)\times \RR^+$, is not acting properly on the Heisenberg group $\cN$, 
$[\lam_{CR}]\neq 0$ in $H^1_d(\Aut_{CR}(\cN),C^\infty(\cN,\RR^+))$.
 On the other hand, if $X=\cN-\{\mbi{0}\}$, then 
$\Aut_{CR}(X)={\rm U}(n)\times \RR^+$ which acts properly on $X$.
Then $[\lam_{CR}]\in H^1_d({\rm U}(n)\times \RR^+, C^\infty(X,\RR^+))= \{0\} $.
The quotient of $X$ by an infinite discrete subgroup of\/  
${\rm U}(n)\times \RR^+$
is an infra-Hopf manifold.
\end{remark}

\section{$CR$-manifolds $X/\pi$ with virtually solvable group $\pi$}  
 \label{sec:CRsolv}

\subsection{Sasaki manifolds and Reeb flow} \label{sec:sasaki} 
A Sasaki structure on $M$ is equivalent with a \emph{standard} 
$CR$-structure. For any $CR$-structure  $(\om, J)$, the Reeb field $\xi$ is defined by the conditions $$ \om(\xi) = 1 \text{ and } d\omega(\xi, \cdot) = 0 . $$  A $CR$-structure is called standard if the flow of the  Reeb field is contained in the pseudo-Hermitian group $\Psh_{CR}(M, \omega)$. Then the  metric $$ g=\om\cdot\om+d\om\circ J$$ is called \emph{Sasaki metric}. 
The group
$\Psh_{CR}(M,\om)$ is thus contained in the isometry group of the Sasaki metric $g$. Note further that the Reeb flow is contained in the center of $\Psh_{CR}(M,\om)$ (compare 
\cite{OK}). 
\smallskip

\paragraph{\em The Reeb field generates $S^1$ on compact manifolds.} 
On a compact $CR$-manifold the Reeb flow always gives rise to a circle action. 

\begin{pro}\label{Reebflow}
Let $(M, \om, J)$ be a closed strictly pseudoconvex
standard $CR$-manifold. Then the Reeb field generates an $S^1$-action on $M$. 
\end{pro}

\begin{proof} By the definition of standard $CR$-structure,
the Reeb field $\xi$ generates a one-parameter group ${\sfA}$ of $CR$-transformations for $(\ker\, \om,J)$. 
Let $\Psh(M,(\om,J))$ be the pseudo-Hermitian group.
Since $\Psh(M,(\om,J))\leq \Iso(M,g)$ for the Sasaki metric $g=\om\cdot\om+d\om\circ J$,
$\Psh(M,(\om,J))$ is compact.
If $\displaystyle\bar {\sfA}$ is the closure of ${\sfA}$ in $\Psh(M,(\om,J))$, then 
$\bar A$ is isomorphic to a $k$-torus $T^k$.
Let $\cT^k$ be the distribution of vector fields for $T^k$ on $M$.
Consider the restriction $\om|_{\cT^k} : \cT^k\ra \RR$. 
Then $\cT^k=\langle \xi\rangle\oplus \ker\, (\om|_{\cT^k})$.
(For this, recall $\om(\xi) = 1$, $d\omega(\xi, \cdot) = 0$.) 
Let $\mbi{u}\in \ker\, (\om|_{\cT^k})$.  Since $\ker\, \om$ is $J$-invariant, $J\mbi{u}\in\ker\, \om$.
As $\mbi{u}\in \cT^k$, $\mbi{u}$ generates a one-parameter group of transformations $\{\varphi_t\}_{t\in \RR}$  holomorphic on $\sfD$. Putting  $p_{-t}=\varphi_{-t}p$,  for a point $p\in M$,
note that $\displaystyle [\mbi{u},J\mbi{u}]=\lim_{t\ra 0}({\varphi_t}_*J\mbi{u}_{p_{-t}}-J\mbi{u}_p)/t=
J\lim_{t\ra 0}({\varphi_t}_*\mbi{u}_{p_{-t}}-\mbi{u}_p)/t=J[\mbi{u},\mbi{u}]=\mbi{0}$.
Let $d\om\circ J$ be the positive definite Levi form 
on $\ker\, \om$, then $\displaystyle 2d\om(\mbi{u},J\mbi{u})=\mbi{u}\om(J\mbi{u})-J\mbi{u}\om(\mbi{u})-\om([\mbi{u},J\mbi{u}])=0$.
Thus $\mbi{u}=\mbi{0}$. It follows ${\sfA}=S^1$.
\end{proof}

\subsection{Aspherical Sasaki manifolds} \label{sec:aspherical_Sasaki}
Now let  $M=X/\pi$ be a  $2n+1$-dimensional closed aspherical manifold with a standard $CR$-structure. Since $M$ is compact, the Reeb field $\xi$ generates an $S^1$-action on $M$ (Proposition \ref{Reebflow}).
Then it follows (\cf \cite{OK}) that
\begin{enumerate}
\item A Sasaki structure $(\om,J)$ induces a K\"ahler structure $(\Om, J)$ on
$W$ such that $d\om=p^*\Om$
and $J$ is an induced complex structure from $(\ker\, \om,J)$ with $p_*J=Jp_*$.
\item The central group extension $\displaystyle 
1\to \RR\cap \pi\to \pi\stackrel{\phi}\lra Q\to  1$ embeds into the pseudo-Hermitian group as in the diagram
\begin{equation}\label{eq:g}\begin{CD}
1@>>>\RR@>>> \Psh(X)@>\phi>> \Iso_h(W)@>>> 1\\
@. @AAA @AAA  @AAA @.\\
1@>>>\RR\cap \pi@>>> \pi @>\phi>> Q@>>> 1
\end{CD}
\end{equation}where the quotient group $Q=\pi \big/\, \RR\cap \pi$ acts effectively and properly discontinuously on $W$ 
as a group of K\"ahler isometries. 
\end{enumerate}
It follows from diagram \eqref{eq:g} that $Q$ acts effectively and properly discontinuously on $W$. \smallskip

\subsection{$S^1$-action on a closed aspherical manifold} \label{sec:circle_action}
Let $M=X/\pi$ be a closed  aspherical manifold with an effective $S^1$-action. Recall (\cf Section \ref{sec:compact_auto}, proof of Theorem \ref{thm:1,2,3}) that, for any $x\in M$, the orbit map ${\rm ev}_x : S^1\ra M$ defined by ${\rm ev}_x(t)=t \cdot x$,  
induces an \emph{injective} homomorphism $${\rm ev}_*:\pi_1(S^1,1)=\ZZ \to  \pi_1(M,x)=\pi ,$$ 
such that ${\rm ev}_*(\ZZ)\leq Z(\pi)$ is contained in the center of $\pi$. This implies that the $S^1$-action  lifts to a proper action of $\RR$ on $X$, where $\pi$ commutes with $\RR$. Since $S^1 = \RR/\ZZ$ and the action is effective, we thus have an equivariant principal bundle on the universal cover $X$ of the form
\begin{equation}\label{eq:prin}\begin{CD}
(\ZZ = \RR\cap \pi, \RR)@>>> (\pi,X)@>>> (Q = \pi\big/ \RR\cap \pi,W = X /\RR) \; .
\end{CD}
\end{equation}
\vskip0.2cm

\paragraph{\em Associated principal bundle} 
 We suppose now that $Q$ admits a torsion-free subgroup of finite index. 
Thus we may choose a torsion-free normal subgroup $Q'$ of finite index in $Q$, such that $W/Q'$ is a closed aspherical manifold. 
We put  $\pi' = \phi^{-1}(Q')$ for the preimage of $Q'$ in $\pi$. Then the  central group extension  
\begin{equation}\label{eq:principale}
\begin{CD}
1@>>> \ZZ = \RR \cap \pi@>>>\pi' @>>> Q' @>>> 1 
\end{CD}\end{equation}
\vspace{0.5ex}
gives rise to a principal circle bundle 
\begin{equation}\label{eq:principalb}
\begin{CD}
 S^1 = \RR/\ZZ  @>>> P = X/\pi' @>p>> B = W/ Q' \ .  
\end{CD}\end{equation}

\subsection{Standard $CR$-structures on circle bundles} \label{sec:standard_circle} 
As in Section \ref{sec:aspherical_Sasaki}, we now suppose that the contractible manifold $X$ has a standard $CR$-structure 
$(\om, J)$. We require further that the Reeb flow generates the principal $\RR$-action in \eqref{eq:principale}. Also we assume  that  the $CR$-structure is preserved by $\pi$, that is,  $\pi \leq \Psh(X,\om)$. And we impose that $\pi \cap \RR = \ZZ$. This  ensures that the principal action of $\RR$ on $X$ descends to an \emph{effective} action of $S^1 = \RR\big/ \ZZ$ on $X/\pi$, as in \eqref{eq:principalb}. \smallskip

Since $\pi \leq \Psh(X,\om)$,  the principal $S^1$-bundle \eqref{eq:principalb} with total space  $$P =X/\pi'$$ inherits  a compatible induced  standard $CR$-structure $(\bar \om, J)$. 
The Reeb field $\xi$ pushes down to the Reeb field $\bar \xi$ on $P$. Since the action of $S^1 = \RR / \ZZ$ is effective, $\bar \xi$ is also the fundamental  vector field of the principal $S^1$-action on $P$. This fact implies that the induced contact form $\bar \omega$ is in fact a \emph{connection form} for the principal circle bundle $P$ in \eqref{eq:principalb}. \smallskip

Furthermore, $d\omega$ is the curvature form of the connection $\bar \om$ and satisfies  $$d \omega = p^* \bar \Omega \, , $$ for a closed form $\bar \Omega$ on $B$. Since it arises as the curvature of the connection form, it follows that the cohomology class of $\bar \Om$ is integral, and that 
$$   e(B) = \, [ \,  \,\bar \Omega \, ]  \; \in H^2(B, \ZZ) $$ 
is the characteristic class of the bundle \eqref{eq:principalb} (\cf \cite{Ko} or \cite[Section 2.2]{Bl}). \smallskip 

Since we also have a $CR$-structure, $\bar \Omega$ is its  associated K\"ahler form on $B$. \smallskip

\paragraph{\em Finite group action on  $P$}
Since the group $\pi$ is contained in $\Psh(X,\om)$, $\pi$ centralizes the $\RR$-action on $X$, since it  arises from the Reeb flow. Therefore $\pi$ acts on $P$ by bundle automorphisms with respect to \eqref{eq:principalb}. Furthermore, the action of $\pi$ on $P$ descends to the group $$\mu = Q/ Q' \, , $$ 
which is acting on $B= W/Q'$ by K\"ahler isometries. That is, $\mu$ preserves the K\"ahler form on $B$, and, in particular, it fixes the K\"ahler class $e(B)$. \smallskip

We further note: 

\begin{lemma}\label{eff}
The holomorphic action of the finite quotient group $\mu$ on $B$ is effective, that is, the homomorphism 
$\displaystyle \mu \to  {\rm hol}(B)$ is injective.
\end{lemma}

\begin{proof}
By our construction, $Q$ is normalizing $Q'$, and it has an effective action on $W$. Since $W \to W/ Q'$ is a covering map, any element of $Q$ that acts trivially on $W/Q'$ is a lift of the identity of $W/Q'$ with respect to this  covering. Therefore, it must be in $Q'$, showing that $\displaystyle \mu \to  {\rm hol}(B)$ is injective.
\end{proof}

\subsection{Biholomorphism between $W$ and $\CC^n$}\label{sec:biho}

By  \cite[Theorem 2.1]{BC}, the aspherical K\"ahler  manifold $W/Q'$  is biholomorphic to
a complex euclidean space form $\CC^n/\rho(Q')$, 
where $\rho: Q'\to  E_{\CC}(n)=\CC^n\rtimes {\rm  U}(n)$
is a faithful representation. Note that $\rho(Q')$ 
is a Bieberbach group. Therefore, $\Lambda =\CC^n\cap \rho(Q')$ is a maximal free abelian normal subgroup of $\rho(Q')$ and of finite index in $\rho(Q')$. Let $$ \displaystyle H: \;  W \to  \CC^n$$ be a corresponding biholomorphism equivariant with respect to $\rho$. From this, we see that (going down to a finite index subgroup if necessary)
we may choose $Q'$ such that $\rho(Q') = \Lambda$ is contained in  $\CC^n$. Then $X/Q'$ is a complex torus biholomorphic to $$ T^n_\CC = \CC^n / \Lambda . $$
Moreover, we have a covering action:
 \begin{equation}\label{eq:cov}\begin{CD}
\Lambda @>>> (Q^H,\CC^n)@>>>(Q^H/\Lambda,T^n_\CC)
\end{CD}\end{equation} where 
$Q^H$ induces a holomorphic action of $
Q^H/\Lambda$ on $T^n_\CC$. That is,
there is a homomorphism $\displaystyle \theta: Q^H/\Lambda\to  {\rm hol}(T^n_\CC)$. Every biholomorphic map of a complex torus $T^n_\CC$ is induced by a complex affine transformation of the vector space $\CC^n$, see \cite{GH}. Thus it follows that $Q^H$ is contained in $ E_{\CC}(n)$. In particular,  $\rho$ extends to a homomorphism $\rho:Q \to  E_{\CC}(n)$ inducing $\theta$. \smallskip

Since $\rho(Q)$ is a crystallographic group, we may, in addition,  arrange things such that $\rho(Q')$ equals $ \Lambda$: 
$$ \rho(Q') = \Lambda = \rho(Q) \cap \CC^n .$$ 
Then, the finite group $Q^H/\Lambda$ maps injectively to  ${\rm U}(n)$,  so that we have
\begin{equation}\label{eq:emfini}
\mu =Q^H/\Lambda\leq {\rm U}(n).
\end{equation} 

\paragraph{\em Compatible choice of associated linear Hermitian form} Consider now the characteristic class $e(B)$ of the circle bundle \eqref{eq:principalb}. By the biholomorphism $B\to T^n_\CC$, $e(B)$ is transported to a class $$e(T^n_\CC) = e(B)^H \in H^2(T^n_\CC, \ZZ) .$$ Moreover, since $e(B)$ is a K\"ahler class, $e(T^n_\CC)$ is contained in the K\"ahler cone of $T^n_\CC$. 

\begin{lemma} \label{lem:linear_class}
There exists a positive definite linear Hermitian two form $\Omega_{\CC^n}$ $($of type $(1,1))$ on 
$\CC^n$ such that its image $\bar \Omega_\CC$ on $T^n_\CC$ represents the  characteristic class $e(T^n_\CC)$. Then,  we have, 
\begin{equation} \label{eq:e}
e(T^n_\CC) =  
[\, \bar \Omega_{\CC^n} \, ] \in H^2(T^n_\CC, \ZZ) \, . \end{equation}
\end{lemma}
\begin{proof} Let $g$ be a K\"ahler metric on $T^n_\CC$ with K\"ahler form $\bar \Omega$ such that the K\"ahler class $e(T^n_\CC) = [\bar \Omega]$. Since $T^n_\CC$ has trivial canonical bundle its first Chern-class $c_1(T^n_\CC)$ vanishes. Let $\Theta$ denote the Ricci form for $g$. Then $0 = c_1(T^n_\CC) = [\Theta]$. 
In particular, $\Theta$ is null-cohomologous. In this situation the Calabi-Yau existence theorem for K\"ahler metrics with prescibed Ricci curvature \cite[Chapter 11]{Be} asserts that there exists a unique Ricci flat K\"ahler metric $g'$ on $T^n_\CC$ with K\"ahler form $\bar \Omega'$,  satisfying $[\bar \Omega'] = [\bar \Omega] = e(T^n_\CC)$. As a consequence of the Cheeger-Gromoll splitting and decomposition theorem for Ricci-nonnegatively curved manifolds \cite{CG, FW} it follows that the only Ricci flat Riemannian metrics on a torus (in fact on closed aspherical manifolds) are flat metrics. This shows that $g'$ is a flat K\"ahler metric on $T^n_\CC$ and thus invariant by the holomorphic action of $T^n_\CC$ on itself. Pulling back $\bar \Omega'$ to a form $\Omega'$ on $\CC^n$, $\Omega_{\CC^n} = \Omega'$ is linear and it has the required properties. 
\end{proof}

Note that by this construction, $\rho(Q)$ preserves the positive definite Hermitian form  $\Omega_{\CC^n}$, and we may thus put $ {\rm U}(n) =  {\rm U}(\Omega_{\CC^n}) $ on $\CC^n$. \smallskip

\paragraph{\em Equivalence of bundles over the complex torus} 
By the identification $W = \CC^n$ via the biholomorphism $H$ and our construction, the bundle \eqref{eq:prin}, now equates to an equivariant  
principal bundle 
\begin{equation}\label{eq:prin2}\begin{CD}
(\ZZ = \RR\cap \pi', \RR)@>>> (\pi',X)@>>> ( \Lambda  = \pi' \big/ \RR\cap \pi' , \CC^n = X /\RR) \; 
\end{CD}
\end{equation}
that descends to a principal circle bundle 
\begin{equation}\label{eq:principalb2}
\begin{CD}
 S^1 = \RR/\ZZ  @>>>P = X/\pi' @>p>> T^n_\CC =  \CC^n/ \Lambda \;  ,  
\end{CD}\end{equation}
\vspace{0.5ex}
whose characteristic class is $e(T^n_\CC)$ as defined in \eqref{eq:e}. 

\subsection{Locally homogeneous standard $CR$-structure on $P$}\label{sec:crl}
For brevity, we put  $\bar \Omega_0 = \bar \Omega_{\CC^n}$ for the above  linear K\"ahler form on $T^n_\CC$. Since the bundle $P$ has fundamental class $e(P) =  [ \bar \Omega_{0}]$,   
there exists a $CR$-structure  $(\bar \eta,J)$ on $P$, with contact form $\bar \eta$,  which is satisfying 
 \begin{equation} \label{eq:eta_Omega} d\bar \eta = p^* \, {\bar \Omega}_0  . \end{equation}
Indeed, let  
$\bar \tau \in \Omega^1(T^n_\CC)$ be any one-form such that $\Omega_0 - \Omega =  d\tau $.  
Then 
\begin{equation} \label{eq:eta_def} 
	\bar \eta \, = \, \bar \omega + p^* \bar \tau 
\end{equation}
satisfies \eqref{eq:eta_Omega}.  
The forms $\bar \om$ and $\bar \eta$ have the same Reeb field, so that the Reeb field $\bar \xi$ generates the original $S^1$-action on $P$. That is,  $\bar \eta$ is a connection form for the principal bundle \eqref{eq:principalb2}.  \smallskip

The $CR$-structure  $(\bar \eta,J)$ on $P$ then pulls back to a $\pi'$-invariant $CR$-structure  $(\eta,J)$ on $X$, with contact form $\eta$, and Reeb field $\xi$, generating the $\RR$-action on $X$. Moreover, we have $d\eta = p^*\Omega_0$. (In particular, all the conditions on $(X, \eta, J)$
 in Section \ref{sec:standard_circle} are satisfied.) We note that the standard $CR$-structure $(\eta,J)$ on $X$ is homogeneous. Indeed,  
 $\Psh(X,\eta)$ acts transitively on $X$, since the holomorphic 
isometries of $(X, \eta, J) \big/ \, \RR = (\CC^n, \Omega_0)$ lift to elements of $\Psh(X,\eta)$ (see  \cite[Proposition 3.4]{OK}).  \smallskip

We show now that $X$ with its $CR$-structure is equivalent to the Heisenberg group 
$\cN$ equipped with its standard left-invariant $CR$-structure. In fact, it  follows that $\Psh(X,\eta) = \cN \rtimes {\rm U}(n)$, as  is noted  in  \cite[Proposition 6.1 (2)]{OK}. Here, the Reeb flow $\RR$ generates  the center of the Heisenberg group $\cN$, and $\cN/\RR =  \CC^n$. Note that, by construction,  $\pi'$ is contained in  
$\Psh(X, \eta,J)$. Therefore   $\pi'$ is a discrete uniform subgroup of $$\Psh(X, \eta,J) = \Psh(\cN) = \cN \rtimes {\rm U}(n) . $$
Note further that  $\pi'$ is a nilpotent group, since it is a central extension of $\Lambda$. Therefore $\pi' \leq \cN$ is a uniform lattice in $\cN$, by the Bieberbach theorem for nilmanifolds \cite{Aus,KLR}.  In particular, the orbit map $\cN \to X$ gives an $S^1$-equivariant diffeomorphism $$\cN/\pi' \to X/\pi'$$ over the base space $T^n_\CC$. This shows that $X/\pi'$ is a Heisenberg nilmanifold. 

\subsection{Locally homogeneous standard $CR$-structure on $X/\pi$}\label{crlquotinent}
 As we have seen, the group $\mu = \pi /\pi'$ acts on the total space $P$ of the principal circle bundle \eqref{eq:principalb2} by bundle automorphisms, and the induced action of $\mu$ on $T^n_\CC$ fixes the curvature form $\bar \Omega_0$.\smallskip 

For the final step in the proof of Theorem 3, we show now that there exists a connection form with curvature $\bar \Omega$ that is fixed by the group $\mu$:   

\begin{pro} \label{prop:invariant_eta}
There exists a connection form $\bar \vartheta$ for the principal circle bundle \eqref{eq:principalb2}, such that 
\begin{enumerate}
	\item  $d\bar \vartheta = p^* \, {\bar \Omega}_0$. 
	\item  $\mu$ is contained in $\Psh(P, \bar \vartheta, J)$. 
\end{enumerate}
\end{pro}
\begin{proof} Let $\cA(\bar \Omega_0) = 
\{ \bar \eta \in \Omega^1(P)  \mid d \bar \eta  = p^* \, {\bar \Omega}_0 \}$ be the space of connection forms for the bundle  \eqref{eq:eta_def} with curvature $\bar \Omega_0$. Choose 
a base point $\bar \eta_0 \in \cA(\bar \Omega_0)$. Then 
$\cA(\bar \Omega_0) = \bar \eta_0 + \{ p^* \tau \mid \tau \in Z^1(T^n_\CC) \}$ is an affine space over the vector 
space of closed one-forms $Z^1(T^n_\CC)$. 
Since the curvature $\bar \Omega_0$ is preserved by $\mu$, for any $g \in \mu$, $\bar \eta \in \cA(\bar \Omega_0)$, it follows $g^* \bar \eta \in \cA(\bar \Omega_0)$. Therefore the  group $\mu$ acts on $\cA(\bar \Omega_0)$. For $g \in \mu$, define 
$$  c( g ) = g^* \bar \eta_0 - \bar \eta_0 \, \in Z^1(T^n_\CC) \; . $$
Since $c(g_1 g_2) = g_2^* g_1^* \bar \eta_0 - \bar \eta_0 = g_2^* c(g_1) + c(g_2)$, $c: \mu \to Z^1(T^n_\CC) $ defines a one-cocycle for the natural representation of $\mu$ on $Z^1(T^n_\CC)$. 
The action of  $\mu$  on $\cA(\bar \Omega_0)$ is by affine transformations, since, for $\bar \eta \in  \cA(\bar \Omega_0)$, we have  $$ g^*   \bar \eta =  \bar \eta_0 + \, g^* (\bar \eta -  \bar \eta_0) + c(g).$$ 
There is a 
corresponding cohomology class 
$[c ] \in H^1( \mu, Z^1(T^n_\CC))$ 
associated to the one-cocycle $c$. Since $\mu$ is finite, we have 
$H^1( \mu, Z^1(T^n_\CC)) = \{0 \}$, so $c$ must be a coboundary. This implies that there exists $\tau \in  Z^1(T^n_\CC)$ such that $c(g) = g^* \tau - \tau$. As a consequence, $\bar \vartheta = \bar \eta_0 - \tau$ is a fixed point for the action of $\mu$.\end{proof}

Pulling back the connection form $\bar \vartheta$ to a connection form $\vartheta$ on the covering bundle \eqref{eq:prin2}, we note:

\begin{corollary}\label{cro:final} There exists on $X$ a connection form $\vartheta$ for the bundle \eqref{eq:prin}, 
with $d \vartheta = \Omega_0$, such that the corresponding $CR$-structure on $X$ satisfies  $$ \pi \leq \Psh(X, \vartheta,J) .$$
\end{corollary}  \smallskip

As above, the $CR$-structure $(X, \vartheta,J)$ is homogeneous and $\Psh(X, \vartheta,J) = \Psh(\cN) = \cN \rtimes {\rm U}(n)$. Therefore $X/\pi$ is an infra-nilmanifold that is finitely covered by the nilmanifold $X/\pi'$. This concludes the proof of Theorem 3.

\section{Indication to the proof of Theorem \ref{Tqc}} \label{sec:qcsolv}

\subsection{Quaternionic contact structures}\label{def:qc}
Let $X$ be a $4n+3$-dimensional smooth manifold.
A \emph{quaternionic contact structure} is a codimension $3$-subbundle $\sfD$ on
$X$ which satisfies that $\displaystyle \sfD\oplus[\sfD,\sfD]=TX$.
In addition the following conditions are required: \smallskip 

There exists a non-degenerate ${\rm Im}\, \HH$-valued $1$-form 
$\om=\om_1i+\om_2j+\om_3k$ called a \emph{quaternionic contact form} on $X$ such that
\begin{enumerate}
\item[(1)] $\displaystyle \ker\,\om=\mathop{\cap}_{i=1}^3\ker\,\om_i=\sfD$. 
\item[(2)] $\displaystyle \om\mathop{\we}\om\mathop{\we}\om\,\we  \overbrace{d\om\we\cdots\we d\om}^n\neq 0$ on $X$.
\end{enumerate}
The non-degeneracy of $d\om_k$, $k=1,2,3$, on $\sfD$ defines the bundle of endomorphisms $\{J_1,J_2,J_3\}$: 
\begin{equation}\label{hypercom}
J_k=(d\om_j\vert_{\sfD})^{-1}\circ
(d\om_{i}\vert_{\sfD}):{\sfD} \ra\, {\sfD}, \ \, (i,j,k)\sim(1,2,3) 
\end{equation}which constitutes a \emph{hypercomplex structure} on $\sfD$. 
Note that the Levi form $$d\om_i\circ J_i: \sfD\times \sfD\lra\, \RR , \; i=1,2,3$$
is a positive definite symmetric bilinear form on $\sfD$.
Then $(X,\sfD,\om,\{J_i\}_{i=1}^{3})$ is said to be a positive definite $qc$-manifold.
\smallskip

\paragraph{\em Standard $qc$-manifolds and three-Sasaki  manifolds}
For any $qc$-manifold $M$, let $\cT =\{\xi_1,\xi_2,\xi_3\}$ denote the three-dimensional integrable distribution complementary to the codimension three subbundle $\sfD$ on $M$ determined by the $qc$-structure, called \emph{$qck$-distribution} (see Section \ref{sec:parabolic_geom}). If $\cT$ generates a subgroup
of ${\Psh}_{qc}(M)$, then $M$ is called a \emph{standard} $qc$-manifold.\smallskip 

\begin{pro}\label{qck-Reebflow}
Let  $(M,\om,\{J_i\}_{i=1}^{3})$ be a closed strictly pseudoconvex
standard $qc$-manifold. Then the Reeb fields generate a
compact Lie group action of\/   $\SU(2)$ or the torus group $T^3$  on $M$. 
\end{pro}

See \cite[Proposition 4.5]{Ka} for the proof. 
If the Reeb fields generate an $\SU(2)$-action then $M$ is called a \emph{three-Sasaki  manifold}  (sometimes also a quaternionic $CR$-manifold \cite{AK}).

\subsection{Aspherical standard $qc$-manifolds}\label{asphestanqc} 
Let $X/\pi$ be a $4n+3$-dimensional positive definite closed aspherical \emph{standard}
$qc$-manifold. Since  $X/\pi$ is aspherical, 
the $qck$-distribution $\hat\cT=\{\hat\xi_1, \hat\xi_2,\hat\xi_3\}$ on $X/\pi$ generates a compact Lie group action by  Proposition \ref{qck-Reebflow}. Since $X/\pi$ is aspherical, the flow is given by a three-torus $T^3\leq {\Psh}_{qc}(X/\pi)$. Moreover, 
the $T^3$-action lifts to a proper $\RR^3$-action on $X$ (as in Section \ref{sec:circle_action} in the case of circle actions).\smallskip

This gives rise to an equivariant principal bundle over $W=X/\RR^3$:
$$\displaystyle (\RR^3\cap \pi, \RR^3)\to  \, (\pi, X)\stackrel{p}\lra\, (Q,W) \; .  $$ 
Then it follows (\cf \cite{Ka}) that\smallskip

\begin{enumerate}
\item The standard qc-structure  $(\om, \{J_i\}_{i=1}^{3})$ induces a hyper-K\"ahler structure $(\Om, J = \{J_i\}_{i=1}^{3})$ on $W$,  such that $d\om=p^*\Om$, where $\Omega = \Om_1i+\Om_2j+\Om_3k$ 
and $J$ is an induced hypercomplex structure from $(\ker\, \om,J)$ with $p_*J=Jp_*$. Here $\Omega_i$ is a K\"ahler form with respect to $J_i$,
$i=1,2,3$.
\item The central group extension $\displaystyle 
1\to \RR^3\cap \pi\to \pi\stackrel{\phi}\lra Q\to  1$ embeds into the pseudo-Hermitian group of the $qc$-structure as in the diagram
\begin{equation}\label{eq:gqc}\begin{CD}
1@>>>\RR^3@>>> {\Psh}_{qc}(X)@>\phi>> {\Iso}_{hK}(W)@>>> 1\\
@. @AAA @AAA  @AAA @.\\
1@>>>\RR^3 \cap \pi@>>> \pi @>\phi>> Q@>>> 1
\end{CD}
\end{equation} where the quotient group $Q=\pi \big/\, \RR^3 \cap \pi$ acts effectively and properly discontinuously on $W$, and  
as a group of hyper-K\"ahler isometries for $(\Omega,J)$. 
\vskip0.1cm
\item[(3)]  $X$ is $qc$-homogeneous if and only if $W$ is hyper-K\"ahler homogeneous (\cite[Proposition C]{Ka}).
\end{enumerate}
\smallskip

\paragraph{\bf Detailed explanation of \eqref{eq:gqc}} 

Let $\om=\om_1i+\om_2j+\om_3k$ be a lift of
the $qc$-form of $X/\pi$ to the universal covering (\cf Section \ref{sec:parabolic_geom}). Then $\om$ is
a $\pi$-invariant non-degenerate ${\rm Im}\, \HH$-valued $1$-form
such that $\pi\leq \Psh_{qc}(X,(\om, \{J_i\}_{j=1}^3)$  as in
 (3) of \eqref{Pshqcr}.
Let $\cT=\{\xi_1, \xi_2, \xi_3\}$ be a lift of
$\hat\cT$ which generates a proper $\RR^3$-action
on $X$. 
 Since $\pi$ centralizes $\RR^3$, for every $\ga\in \pi$,
 it follows $\ga_*\xi_i=\xi_i$ $(i=1,2,3)$.
 Noting the action of $\pi$ on $\om$ from
(3) of \eqref{Pshqcr},  each element $\ga\in \pi$ satisfies
\begin{equation}\label{eq:actionpi}
\ga^*\om=\om, \ \ga_*J_i=J_i\ga_*.
\end{equation}
Recall that the 
$qc$-structure of $\displaystyle(X,\RR^3, \{\om_i, J_i\}_{i=1}^3, g_\om)$ 
induces a simply connected complete hyper-K\"ahler manifold
 $\displaystyle(W, \{\Om_{i},J_i\}_{i=1}^3, g)$ for which
$p^*\Om=d\om$ where $\om=\om_1i+\om_2j+\om_3k$ is the $\pi$-invariant one-form
on $X$ and $\Om$ is a $Q$-invariant two-form on $W$: 
 \begin{equation}\label{eq:hyper-K}
 \Om=\Om_1i+\Om_2j+\Om_3k.
 \end{equation}
Furthermore,  $g_\om=\sum_{i=1}^3\om_i\cdot\om_i+d\om_1\circ J_1$ is
a canonical Riemannian metric  on $X$ and $g=\Om_i\circ J_i$ $(i=1,2,3)$
is a hyper-K\"ahler metric on $W$.
By the equation $p^*\Om=d\om$ and \eqref{eq:actionpi}, each $\al\in Q$ 
preserves the hyper-K\"ahler struture on $W$: 
\begin{equation}\label{eq:3hquat}
\al^*\Om=\Om,\ \, \al_*J_i=J_i\al_*.
\end{equation}In paricular, $Q$ acts as K\"ahler isometries of $(W,(\Om_i,J_i))$.
\smallskip

\subsubsection{Associated principal three-torus bundle} 
 We suppose now that $Q$ admits a torsion-free subgroup of finite index. 
Thus we may choose a torsion-free normal subgroup $Q'$ of finite index in $Q$, such that $W/Q'$ is a closed aspherical manifold. 
We put  $\pi' = \phi^{-1}(Q')$ for the preimage of $Q'$ in $\pi$. Then the  central group extension  
\begin{equation}\label{eq:principaler3}
\begin{CD}
1@>>> \ZZ^3 = \RR^3 \cap \pi@>>>\pi' @>>> Q' @>>> 1 
\end{CD}\end{equation}
\vspace{0.5ex}
gives rise to a principal torus bundle 
\begin{equation}\label{eq:principalt3b}
\begin{CD}
 T^3 = \RR^3/\ZZ^3  @>>> P = X/\pi' @>p>> B = W/ Q' \ .  
\end{CD}\end{equation}

\subsection{Proof of Lemma D}
Let $g$ be hyper-K\"ahler  with hyper-K\"ahler form $\Omega$ (compare \eqref{eq:hyper-K}) on $M$. Since $g$ is hyper-K\"ahler its holonomy group is contained in $\Sp(n)$. In particular, it is contained in $\mathrm{SU}(n)$, which implies that $g$ is a Ricci- flat K\"ahler metric on $M$ (e.g.\ \cite[Proposition 10.29]{Be}). In view of the fact that $M$ is aspherical all Ricci flat metrics are flat \cite{FW}. We therefore have established that 
$g$ is a flat hyper-K\"ahler metric. 
In particular, its universal covering space must be 
isometric to $\HH^n$ with linear K\"ahler structure $\Omega$. 

\vskip0.1cm
\subsection{Hypercomplex isometric isomorphim} 
Applying Lemma D in the introduction to the hyper-K\"ahler manifold  $B = W/Q'$ asserts that $B = W/Q'$ with the hyper-K\"ahler structure \eqref{eq:hyper-K} 
is hyper-holomorphically isometric to a flat hyper-K\"ahler torus $T^n_\HH =\HH^n/\Lambda$,  where $\Lambda \leq \HH^n$ is a lattice. 

\subsection{Conclusion of the proof} 
Using basic methods developed in \cite{Ka} the remaining part of the proof of Theorem \ref{Tqc} follows a similar but simplified procedure as that of Section \ref{sec:CRsolv}. The key step is to establish that the universal covering $X$ with its $qc$-structure is homogeneous and  arises from a quaternionic Heisenberg group, that is $X$ is $qc$-isometric to a $qc$-Heisenberg group with its standard $qc$-structure.  The details shall be presented elsewhere. 

\end{document}